\newif\ifarxiv\arxivfalse

\RequirePackage[dvipsnames]{xcolor}
\ifarxiv
	\documentclass[10pt, twoside, reqno]{amsart}
\else
	\RequirePackage{amsmath}
	\documentclass[smallextended]{svjour3}
	\usepackage[11pt]{extsizes}
	\usepackage[a4paper, total={15cm, 23cm},left=2.5cm,right=2.5cm,bottom=3cm,top=3cm]{geometry}
\fi
	\PassOptionsToPackage{noend}{algpseudocode}
	\PassOptionsToPackage{bookmarksdepth=4}{hyperref}
	
	\ifarxiv\else
		\makeatletter
			\let\cl@chapter\undefined
		\makeatother

		\smartqed
	\fi

	\usepackage{todonotes}
	\usepackage{bm}
	\usepackage{booktabs}
	\usepackage{multirow}
\usepackage[linesnumbered]{algorithm2e}

\SetInd{0.5em}{1em} 
    \usepackage[%
		eqreset=section,
		thmreset=section,
	]{myPreamble}
    \usepackage{enumitem}
 
	\makeatletter
		\newcommand\mynum[1]{$ˆ{\@fnsymbol{#1}}$}
	\makeatother
	
	\CreateNewTheorem[Fact]{fact}[dummythm]{Fact}
	\ifarxiv\else
	    
	\fi
	
	\newcommand\C{\mathcal{C}}

	\makeatletter
		\newcommand\D[1][h]{%
			\@ifnextchar_{\operatorname D}{%
				\ifstrempty{#1}{\operatorname D}{\operatorname D_{#1}}%
			}
		}
	\makeatother
	
\Crefname{equation}{}{}

\newcommand{\TheKeywords}{%
	Nonsmooth nonconvex optimization,
        Relatively weakly convex functions, 
        Projected subgradient methods,
        Linear convergence.%
}
\newcommand{\TheSubjclass}{%
	90C06, 
	90C25, 
	90C26, 
	49J52, 
	49J53.
}
\newcommand{\TheTitle}{%
	Relative Weak Convexity and Projected Subgradient Methods: Analysis and Convergence%
}
\newcommand{\TheShortTitle}{%
	Relative Weak Convexity and Projected Subgradient Methods%
}
\newcommand{\TheShortAuthor}{%
	M. Rahimi and M. Ahookhosh%
}
\newcommand{\TheFunding}{%
	We acknowledge the support by the \emph{Research Foundation Flanders (FWO)} research project G081222N and
	\emph{UA BOF DocPRO4 projects} with ID 46929 and 48996.
}
\newcommand{\TheAddressUA}{%
		Department of Mathematics, University of Antwerp, Middelheimlaan 1, B-2020 Antwerp, Belgium}
\newcommand{\TheAbstract}{%
We introduce the class of relatively weakly convex functions, which extends the classical notion of weak convexity by measuring nonconvexity relative to a distance-generating function. We investigate the fundamental properties of this function class, establishing characterization results, calculus rules, and illustrative examples. We further analyze the associated optimization landscape and identify a neighborhood of the set of global minimizers that is free of saddle points. Motivated by this geometric structure, we propose the Projected SubGradient Algorithm (PSGA) with several step-size strategies. Under a sharpness error bound, we prove that, when initialized within this saddle-point-free neighborhood, the iterates generated by PSGA converge to a global minimizer for each of the proposed step-size strategies. Furthermore, linear convergence is established for the geometrically decaying step-size strategy.
}

\begin{document}

\ifarxiv
	\title[\TheShortTitle]{\TheTitle}
	\author[\TheShortAuthor]{%
		Morteza Rahimi\textsuperscript{1},\ 
		Masoud Ahookhosh\textsuperscript{2},\ and\ 

	}
	\thanks{\textsuperscript{1}\TheAddressKU}
	\thanks{\textsuperscript{2}\TheAddressMons}
	\thanks{\TheFunding}
	\keywords{\TheKeywords}
	\subjclass{\TheSubjclass}


		\begin{abstract}
			\TheAbstract
		\end{abstract}

		\maketitle
        
\else

	\journalname{....}

	\title{\TheTitle\thanks{\TheFunding}}
	\titlerunning{\TheShortTitle}

	\author{%
            Morteza Rahimi\and
		Masoud Ahookhosh%
	}
	\authorrunning{\TheShortAuthor}
    
    \institute{%
		M. Rahimi and M. Ahookhosh
		\at
			\TheAddressUA.
			{\tt
				\{%
					\href{mailto:morteza.rahimi@uantwerpen.be}{morteza.rahimi},%
					\href{mailto:masoud.ahookhosh@uantwerpen.be}{masoud.ahookhosh}%
				\}%
				\href{mailto:morteza.rahimi@uantwerpen.be, masoud.ahookhosh@uantwerpen.be}{@uantwerpen.be}%
			}%
	}%

	\maketitle

	\begin{abstract}
		\TheAbstract
		\keywords{\TheKeywords}%
		\subclass{\TheSubjclass}%
	\end{abstract}
\fi

\vspace{-5mm}
\section{Introduction}\label{sec:Introduction}
Let us consider the nonsmooth and nonconvex optimization problem
\begin{equation}\label{problem}
        \min_{x\in X} f(x),
\end{equation}
under the following standing assumption:

\begin{ass}[Basic assumptions]\label{ass:basic} For problem \cref{problem}, we assume:
\begin{description}
\item[(a)] The function \(\func{h}{\R^n}{\Rinf}\) is a Legendre function;
\item[(b)] The function $\func{f}{\R^n}{\Rinf=\R\cup \{+\infty\}}$ is proper, nonsmooth, and weakly convex relative to the distance-generating function $h$ (see \Cref{Rsharpbound} for details);
\item[(c)]  $X\subseteq \interior\dom(h)$ is a nonempty, closed, and convex set representing our feasible set;
\item[(d)] The function $f$ admits a sharpness error bound with constant $\mu$ (see \Cref{Rsharpbound});
\item[(e)] The set of minimizers $\mathcal{X}^*:=\argmin_{x\in \R^n} f(x)$ is nonempty, with $f^*>-\infty$ denoting the optimal value.
\end{description}
\end{ass}

Nonsmooth optimization problems are ubiquitous in modern scientific and engineering applications, arising naturally from the interplay between data fidelity terms and regularization mechanisms. For instance, in machine learning, important examples include support vector machines with hinge-loss objectives, sparse learning models such as the Lasso based on $\ell_1$-regularization, matrix completion and low-rank recovery problems involving nuclear norms, and deep neural networks with ReLU and max-pooling operations; see, e.g., \cite{ahookhosh2019optimal,beck2017first,candes2009exact,goodfellow2016deep,shalev2014understanding,tibshirani1996regression}. Nonsmoothness in the data fidelity term often originates from robust or margin-based losses, including hinge, absolute-value, quantile, and Huber-type losses, which improve robustness to outliers and enhance predictive performance; see \cite{beck2017first,huber1992robust,shalev2014understanding}. In regularization, nonsmooth penalties are widely employed to enforce structural properties such as sparsity, group sparsity, low rank, and piecewise smoothness through $\ell_1$-, mixed-, nuclear-, and total variation norms; cf. \cite{beck2017first}. These models are fundamental in many fields of our scientific communities, motivating efficient algorithms for nonsmooth optimization.

Numerous first-order algorithms have been proposed for nonsmooth optimization. However, many effective schemes, including proximal-point, proximal gradient, Douglas–Rachford splitting, and primal–dual methods, require the objective function to exhibit favorable structures, such as separability, composite smooth or nonsmooth forms, or efficiently computable proximal mappings; cf. \cite{beck2017first,combettes2011proximal}. In many applications, however, the objective function may lack sufficient structure for these methods to be directly applicable or computationally efficient. By contrast, subgradient methods remain broadly applicable under minimal assumptions, requiring only the ability to compute a subgradient of the objective function; see, e.g., \cite{ahookhosh2018optimal,ahookhosh2017optimal,gaudioso2022essentials,neumaier2016osga}. Their simplicity, low per-iteration computational cost, and modest memory requirements make them particularly attractive for large-scale and high-dimensional problems where more sophisticated structured methods may be impractical.

Although the convergence theory of subgradient methods is well developed in the convex setting (e.g., \cite{bertsekas1999nonlinear,nedic2001incremental,nesterov2018lectures,nesterov2009primal,polyak1969minimization,polyak1987introduction,shor2012minimization}), considerably less is known in the nonconvex regime, especially regarding rates of convergence. Existing results guarantee convergence for several important classes of nonconvex functions, including weakly convex \cite{davis2018subgradient,li2021weakly}, paraconvex \cite{rahimi2024projected}, path-differentiable \cite{bolte2022long}, quasiconvex \cite{kiwiel2001convergence,hu2020convergence,quiroz2015inexact}, and tame functions \cite{davis2020stochastic}. However, many practical models arising in applications do not satisfy these structural assumptions. Moreover, even when convergence of existing methods can be guaranteed, explicit complexity estimates and quantitative convergence rates are often unavailable. This motivates the development of a broader theoretical framework for establishing both convergence guarantees and convergence rates of subgradient methods over more general classes of nonsmooth and nonconvex optimization problems, which is the focus of this work.

\vspace{-4mm}
\subsection{{\bf Contribution}}
Our main contributions are summarized as follows:
\begin{description}
  \item[{\bf (i)}] {\bf Class of relatively weakly convex functions.}
  We introduce the class of relatively weakly convex functions, which generalizes classical weak convexity by measuring nonconvexity relative to a reference geometry induced by a distance-generating function. This framework allows us to capture a broader range of nonsmooth and nonconvex objectives arising in real-world applications. We provide a detailed characterization of this class (i.e., \Cref{pro:Equi-RWC} and \Cref{pro:chara-RWC}), along with a rich set of examples illustrating its relevance in applications. We further develop a calculus for relatively weakly convex functions, establishing rules for preserving the property under common operations such as sums, compositions, and perturbations (i.e., \Cref{pro:relWeakConvCal}), and we show the composition of a convex function with a relatively Bregman smooth Jacobian mapping is relatively weakly convex (i.e., \Cref{pro:composition-RWC}). In addition, we study the geometric landscape of these functions and show that, under suitable conditions, there exists a tubular neighborhood around the set of minimizers that is free of spurious local minima and saddle points (i.e., \Cref{pro:dis2}).
  \item[{\bf (ii)}] {\bf Convergence analysis of subgradient methods.}
  We analyze the convergence of subgradient methods under several step-size rules, including constant and diminishing step-sizes. Under a sharpness assumption and assuming the iterates are initialized within a tubular neighborhood of the solution set, we show the convergence (i.e., \Cref{thm:Conv-ND}, \Cref{thm:Conv-SSN}, and \Cref{thm:DS}) and establish linear convergence rates (i.e., \Cref{thm:PSGA-C1} and \Cref{thm:GDS}). These results demonstrate that, despite the lack of smoothness and convexity, appropriately designed subgradient schemes can achieve fast local convergence when combined with suitable geometric properties of the objective function.
\end{description}

\vspace{-4mm}
\subsection{{\bf Organization}}
The remainder of the paper is organized as follows. \Cref{sec:Preliminaries} is devoted to preliminaries and facts necessary in the other sections. \Cref{sec:relweakconv} introduces the class of relatively weakly convex functions and presents its characterization, calculus, and properties. \Cref{sec:subGradMethod} presents the subgradient method and its convergence analysis.
Finally, \Cref{sec:conclusion} delivers our conclusion.

\section{Preliminaries}\label{sec:Preliminaries}

\subsection{{\bf Notation}}

In this paper, we denote the standard {\it inner product} and the {\it Euclidean norm} in $n-$dimensional real {\it Euclidean space}
$\R^n$ by $\langle \cdot,\cdot\rangle$ and 
$\|\cdot\|=\sqrt{\langle \cdot, \cdot\rangle}$, respectively. For a real $m \times n$ matrix $A = [a_{ij}] \in \mathbb{R}^{m\times n}$ the {\it Frobenius norm} is $\Vert A\Vert_F =\sqrt{\sum_{i,j} \vert a_{ij}\vert^2}$.
The set of {\it natural numbers} is denoted by $\N$. The notion $x^T$ represents the {\it transpose} of a vector $x\in \mathbb{R}^n$.
The open ball centered at $x\in\R^n$ with radius $r>0$ is expressed as $\mathbb{B}(x;r)$.
The {\it interior} and {\it closure} of a set $S\subseteq \mathbb{R}^n$ are denoted by 
$\interior S$ and $\cl S$, respectively.
The {\it Euclidean distance} from a point $x\in\R^n$ to a nonempty set $S\subseteq \mathbb{R}^n$ is defined as $\dist(x,S)=\inf_{z\in S} \|z-x\|$. Moreover, the {\it Euclidean projection} of the point $x$ onto $S$ is given by
$\operatorname{proj}_S(x):=\argmin_{z\in S} \|z-x\|$.

For a given function $\func{\varphi}{\R^n}{\Rinf=\R\cup\{+\infty\}}$, the {\it effective domain} of
$\varphi$ is defined as $\dom(\varphi) := \{x \in \R^n : \varphi(x) < +\infty\}$.
The function $\varphi$ is said to be {\it proper} if $\dom(\varphi) \neq \emptyset$.
The convex conjugate of $\varphi$ is given by
$\conj\varphi(x) := \sup_{z\in\R^n} \{ \innprod{x}{z} -\varphi(z)\}$.
The {\it indicator} function of a nonempty set $S\subseteq \mathbb{R}^n$, $\func{\delta_S}{\R^n}{\Rinf}$, is defined as $\delta_{S}(x)=0$ if $x\in S$, and $\delta_{S}(x)=+\infty$ otherwise.

In this paper, we consider a class of functions that we subsequently prove to be subdifferentially regular \cite[Definition 2.3.4]{clarke1990optimization}. This property ensures that the standard subdifferentials in variational analysis \cite{rockafellar2011variational,clarke1990optimization,mordukhovich2006variational} coincide. Hence, we work throughout with the Fr\'{e}chet subdifferential.
Let $\func{\varphi}{\R^n}{\Rinf}$ be a proper function that is locally Lipschitz at $x\in \dom(\varphi)$. The {\it Fr\'{e}chet subdifferential} of $\varphi$ at $x$, denoted by $\partial \varphi(x)$, is the set of all vectors $\zeta\in \R^n$ satisfying
$$\varphi(y) \ge \varphi(x) + \innprod{\zeta}{y-x} + o(\|y-x\|) \quad\text{as}\quad y \to x.$$
Given a nonempty set $S\subseteq \dom(\varphi)$, a point $\ov x\in S$ is said to be a {\it stationary} point of the problem $\displaystyle\min_{x\in S} \varphi(x)$ if
$0\in \partial (\varphi+\delta_S)(\ov x)$.

\subsection{{\bf Bregman distance}}

\begin{defin}[{\bf Legendre function}]\label{def:kernel}%
    Let \(\func{h}{\R^n}{\Rinf}\) be a proper, lsc, and convex function.
	It is said to be
	\begin{enumerateq}
        \item
            a \DEF{distance-generating (kernel) function}, if \(\interior\dom h\neq\emptyset\) and $h\in\C^1(\interior\dom h)$
	\item
		\DEF{suppercoercive}, if
            $\lim_{\|x\|\to\infty}\tfrac{h(x)}{\|x\|}=\infty$;
	\item
		\DEF{essentially smooth}, if $h$ is differentiable on
            \(\interior\dom h\neq\emptyset\) and $\|\nabla h(x_k)\|\to\infty$ for every sequence $\seq{x_k}\subseteq\interior\dom h$ converging to a boundary point of $\dom h$;
	\item
		\DEF{a Legendre kernel function} if it is essentially
            smooth and strictly convex.
	\end{enumerateq}
\end{defin}

\begin{rem}\label{rem:legendre}~
        It is better to mention that a kernel function $\func{h}{\R^n}{\Rinf}$ is classified a Legendre kernel if it satisfies two primary properties: (i) essentially smooth; (ii) essential strict convexity (strict convexity in
 every convex subset of \(\dom(\partial h)\)) which ensures that $\argmin h$ is a singleton. However, under the essential smoothness property, essential strict convexity of $h$ is equivalent to strict convexity on \(\interior\dom(\partial h)\). 
		Additionally, the function $h$ is a Legendre kernel if and only if so is $h^\#$. In this case, the gradient $\nabla h$ forms a bijective mapping between \(\interior\dom h\) and \(\interior\dom\conj h\), satisfying:
        $\dom(h) = \interior \dom(h),$ $\range \nabla h = \dom(\nabla\conj h) = \interior \dom(\conj h),$
        \((\nabla h)^{-1}=\nabla\conj h,\) and \(\conj h(\nabla h(x))=\innprod{x}{\nabla h(x)}-h(x),\) for all \(x\in\interior \dom(h).\)
        Furthermore, if $h\in \C^2$ is a Legendre kernel function and $\nabla^2 h\succ 0$ on $\interior \dom(h)$, then $\conj h\in \C^2$, \cite{rockafellar1977higher}.
\end{rem}

\begin{defin}[{\bf Bregman distance}]
     For a kernel function \(\func{h}{\R^n}{\Rinf}\), the \DEF{Bregman distance} $\func{\D}{\R^n\times\R^n}{\Rinf}$ is given by
	\begin{equation*}\label{eq:bregman}
		\D(x,y)
	{}\coloneqq{}
		\begin{ifcases}
			h(x)-h(y)-\innprod{\nabla h(y)}{x-y}~~ &~ y\in\interior\dom (h)
		\\
			\infty\otherwise.
		\end{ifcases}
	\end{equation*}
\end{defin}

The function \(\D(x,y)\) measures the proximity between \(x\) and \(y\). As some basic properties, \(\D\in\C^0\Big(\dom(h)\times\interior\dom(h)\Big)\) and the function $\D(\cdot,y)$ is differentiable with $\tfrac{\partial \D}{\partial x}(x,y)=\nabla h(x) -\nabla h(y)$.  The function \(h\) is convex if and only if \(\D(x,y)\geq 0\) for all 
\(x\in\dom (h)\) and \(y\in\interior \dom (h)\).
If the function \(h\) is convex, then the function $\D(\cdot,y)$ is proper, lsc, and convex on \(\interior \dom (h)\).
If \(h\) is strictly convex, then \(\D(x,y)= 0\) if and only if \(x=y\).
If the function \(h\) is essentially strictly convex, then $\D(\cdot,y)$ is coercive and essentially strictly convex on \(\interior \dom (h)\).

Let us now define a measure for the lack of symmetry for Bregman distances. 

\begin{defin} [{\bf Symmetry for Bregman distances}]
\label{def:symmRelSmooth}
    Let \(\func{h}{\R^n}{\Rinf}\) be a Legendre function. The {\it symmetry coefficient} of the Bregman distance $\D$ is given by 
    \begin{equation*}
    \theta_h:= \inf \set{\tfrac{\D(x,y)}{\D(y,x)} ~\mid~ (x,y)\in  \interior\dom(h) \times \interior\dom(h),~ x\neq y}, 
    \end{equation*}
    which satisfies $\theta_h\in [0,1]$.
\end{defin}

It comes directly from this definition that
\begin{equation}\label{eq:symmetery}
    \theta_h \D(x,y) \leq \D(y,x)\leq \theta_h^{-1} \D(x,y) \quad \forall x,y\in \interior\dom(h),
\end{equation}
where we adapt the convention $\tfrac{1}{0}=+\infty$ and $+\infty\times \alpha=+\infty$ 
for $\alpha\geq 0$.

The following results are a direct consequence of the definition of the Bregman distance.

\begin{fact}[{\bf Properties of Bregman distances}]\label{fact:BD}
    Let \(\func{h}{\R^n}{\Rinf}\) be a kernel function and $S\subseteq \interior\dom(h)$ be a nonempty convex set.
    If the function 
    \(h\) has \(\nu\)-H\"older continuous gradient with exponent $\nu\in (0,1]$ and constant $G_{h}>0$ on $S$, i.e.,
    \[
        \|\nabla h(y)- \nabla h(x)\| \leq G_{h}\|y-x\|^{\nu},\quad \forall x,y\in S,
    \]
    then \(\D(y,x) \le \tfrac{G_{h}}{1+\nu}\|y-x\|^{1+\nu}\) for all \(x,y\in S\).
\end{fact}

\subsection{{\bf Sharpness error bound condition}}

An error bound is a crucial tool in optimization for achieving faster convergence rates in algorithms. In this study, we focus on a specific type of error bound, named the sharpness error bound.

\begin{defin}[{\bf Sharpness error bound}]\label{Rsharpbound}
    A function $\func{\varphi}{\R^n}{\Rinf}$ is said to admit a sharpness error bound with constant $\mu > 0$ on a nonempty set $S \subseteq \R^n$ if
    \begin{equation*}\label{eq-holerrbound2}
        \mu \dist(x; \mathcal{X}^*) \leq \varphi(x) - \varphi^{*}, \quad \forall x \in S,
    \end{equation*}
    where $\varphi^* := \inf_{x \in S} \varphi(x)$ and the set of minimizers $\mathcal{X}^* := \{x \in S \,:\, \varphi(x) = \varphi^{*}\}$ is nonempty.
    \end{defin}

\section{{\bf Relatively weakly convex function}}\label{sec:relweakconv}
In this section, we introduce the class of relatively weakly convex functions and investigate their characterizations, representative examples, calculus properties, and optimization landscape. Notably, we establish the existence of a neighborhood around the solution set that is free of saddle points.

Let us begin with the definition of relative weak convexity, as an extension of weak convexity.

\begin{defin}[\textbf{Relative weak convexity}]
\label{def:relSmooth}
Let \(\func{h}{\R^n}{\Rinf}\) be a convex function, let $S\subseteq \R^n$ be a convex nonempty set, and let $\rho> 0$. A function \(\func{\varphi}{\R^n}{\Rinf}\) is said to be
\DEF{$\rho$-weakly convex relative to $h$ on $S$} if \(\varphi+\rho h\) is convex on $S$.
When the function $h$ is clear from the context, we may refer to $\varphi$ as relatively $\rho$-weakly convex, or simply as weakly convex relative to $h$ (relatively weakly convex) when omitting the constant $\rho$.
\end{defin}

Relatively weakly convex functions constitute an important class with broad applicability. In the following, we present several representative examples that arise in different settings.

\begin{exa}[{\bf Representative relatively weakly convex functions}] We next mention several classes of relatively weakly convex functions. 
    \begin{enumerateq}
        \item
            In the literature, a specific case of relative weak convexity has been identified based on the choice of $h$. 
            Indeed, when $h=\|\cdot\|^2,$  the concept of relatively weakly convexity corresponds to weak convexity \cite{rolewicz1979paraconvex,vial1983strong}, i.e. a function \(\func{\varphi}{\R^n}{\Rinf}\) is said to be \DEF{\(\rho\)-weakly convex} with \(\rho>0\) if the function \(\varphi(\cdot)+\rho \|\cdot\|^2\) is convex.
        \item
            Another example of relatively weakly convex functions lies in the class of paraconvex functions \cite{rolewicz1979paraconvex,jourani1996open,rahimi2024projected}, i.e., A function $\func{\varphi}{\R^n}{\Rinf}$ is said to be $\nu$-paraconvex function on a convex set $S\subseteq \dom(\varphi)$ for some $0<\nu\leq 1$ if there exists $\rho \geq 0$ such that for any $x,y \in S$ and $\lambda \in [0,1]$,
            \begin{equation}\label{eq:para1}
                \varphi(\lambda x+(1-\lambda)y)\leq \lambda \varphi(x) +(1-\lambda) \varphi(y) + \rho \min\{\lambda , 1-\lambda\} \Vert x-y\Vert ^{1+\nu}.
            \end{equation}
            In particular, Rahimi et al. \cite[Proposition 3.6]{rahimi2024projected} demonstrated that if
             the function $\varphi(\cdot)+\rho \Vert \cdot\Vert ^{\nu+1}$ is convex, i.e., the function $\varphi$ is $\rho$-weakly convex relative to $h=\|\cdot\|^{\nu+1}$ with $\nu\in (0,1]$ and $\rho>0$, then $\varphi$ is $\nu$-paraconvex on $\interior \dom(h)$.
        \item
            One notable example of relatively weakly convex functions arises in the context of difference weakly convex (DWC) optimization problems, formulated as:
            \begin{equation}\label{eq:DCporblem}
                \min_{x\in S} \varphi(x):=g(x)-h(x),
            \end{equation}
            where the functions $g,h:\R^n\to \R$ respectively are $\rho_g$ and $\rho_h$-weakly convex and $h$ is the kernel function. Therefore, the function \(\varphi\) is $\rho$-weakly convex relative to $\ov h$ at which $\rho=\max\{\rho_g, \rho_h\}$ and $\ov h(\cdot)= h(\cdot)+\rho_h \|\cdot\|^2$.
        \item
            Let us consider the class of composite optimization problems:
            \begin{equation}\label{eq:compositeporblem}
                \min_{x\in S} \varphi(x):=\psi(F(x)),
            \end{equation}
            where $\psi:\R^n\to \R$ is a convex and Lipschitz function, and the mapping $F:\R^m\to \R^n$ is relatively Bregman Jacobian smooth on the nonempty convex set $S$; see \Cref{def:rbjs}.
            It can be indicated that the function $\varphi=\psi \circ F$ is relatively weakly convex relative; see \Cref{pro:composition-RWC}.
        \item
            Any relatively Bregman gradient smooth function $\func{\varphi}{\R^n}{\R}$ is a relatively weakly convex function.
    \end{enumerateq}
\end{exa}

\begin{exa}\label{exa:example_saddle1}
    Let us consider the function $\func{f}{\R^2}{\R}$ given by \(f(x,y)= x^2 + y^2 -1.5 y^{\tfrac{4}{3}}\).
    We compute
    \[
    \nabla f(x,y) = \begin{bmatrix}
        2x\\
        2y - 2 y^{\tfrac{1}{3}}
    \end{bmatrix},\quad\quad
    \nabla^{2} f(x,y) = \begin{bmatrix}
        2  &  0\\
        0  & 2 - \tfrac{2}{3y^{\tfrac{2}{3}}}
    \end{bmatrix}.
    \]
    It is evident that the function $f$ is not a convex function possessing a saddle point.
    Specifically, the points \((0,\pm 1)\) and \((0,0)\) are stationary points: \((0,\pm 1)\) are global minima, while \((0,0)\) is a saddle point. This is illustrated by the behavior of \(f\) in different directions near \((0,0)\): along $y=0$ we have $f(x,y)=x^2>0$, while along $x=0$ we get \(f(x,y)= y^2 -1.5 y^{\tfrac{4}{3}}<0\) for $0<|y|<\sqrt{3.375}$.
    The Hessian matrix further confirms that \((0,0)\) is a degenerate saddle point, as one eigenvalue tends to \(-\infty\) as \(y \to 0\). Moreover, the function \(f\) is \(1.5\)-weakly convex relative to \(h(x, y) = y^{\tfrac{4}{3}}\). It is also worth noting that any ball of radius \(r < 1\) centered at a global minimizer contains no saddle points. \Cref{fig:weakconvExa1} illustrates the contour and surface plots of the function.
    \begin{figure}[htp]
        \centering
        \subfloat{\includegraphics[width=7.5cm]{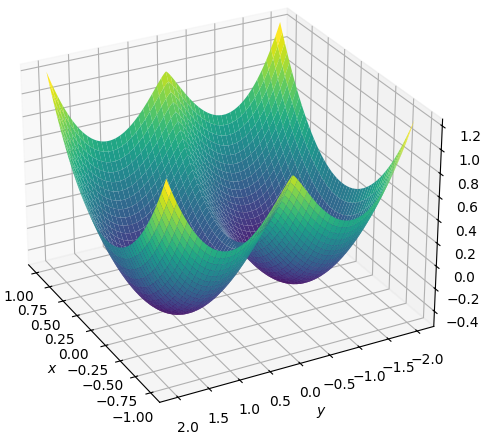}}%
        \qquad
        \subfloat{\includegraphics[width=7.5cm]{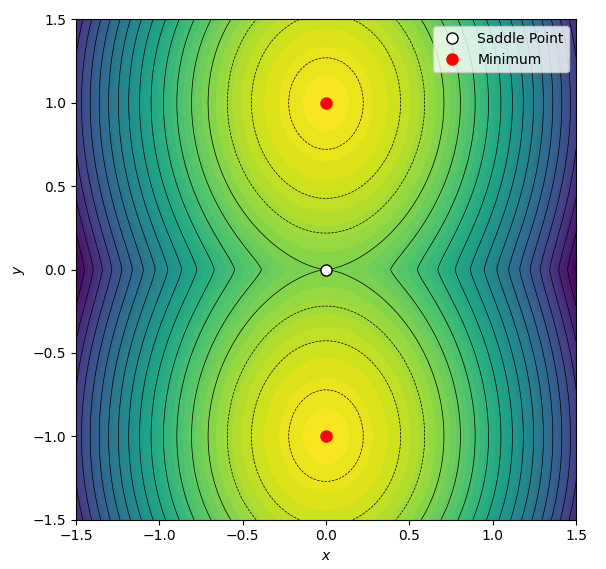}}%
        \caption{Stationary points of the function $f(x,y)= x^2 + y^2 -1.5 y^{\tfrac{4}{3}}$: global minima at \((0,\pm 1)\) (red points) and a saddle point at \((0,0)\) (white point).}\label{fig:weakconvExa1}
    \end{figure}
\end{exa}

The following proposition is a direct result of the definition of relative weak convexity.

\begin{prop}[\textbf{Characterizations I of relative weak convexity}]\label{pro:Equi-RWC}
    Let \(\func{h}{\R^n}{\Rinf}\) be a convex function, let $S\subseteq \interior \dom(h)$ be a convex nonempty set, and let $\rho> 0$. Then, the following assertions are equivalent:
    \begin{enumerateq}[label=(\alph*), ref=(\alph*)]
        \item\label{pro:Equi-RWC:1}
            the function \(\func{\varphi}{\R^n}{\Rinf}\) is $\rho$-weakly convex relative to $h$ on $S$;
        \item\label{pro:Equi-RWC:2}
            for a given $y\in \interior \dom(h)$, the function \(x\mapsto \varphi(x)+\rho\D(x,y)\) is convex on $S$;
        \item\label{pro:Equi-RWC:3}
            for any $x_1,x_2 \in S$ and any $\lambda\in[0,1]$ one holds
            \begin{equation}\label{eq:weakconvexity}
                \varphi(\lambda x_1+(1-\lambda)x_2) \leq \lambda \varphi(x_1) + (1-\lambda) \varphi(x_2) + \rho\Big(\lambda\D\big(x_1,\lambda x_1+(1-\lambda)x_2\big) +(1-\lambda)\D\big(x_2,\lambda x_1+(1-\lambda)x_2\big) \Big).
            \end{equation}
    \end{enumerateq}
\end{prop}
\begin{proof}
    It is clear that $\rho$-weak convexity of $\varphi$ relative to $h$ on $S$ is equivalent to the convexity of function $x\mapsto\varphi(x)+\rho h(x) -\rho h(y) -\rho \innprod{\nabla h(y)}{x-y}$ for a fixed $y\in \interior \dom(h)$. Thus, Assertions \ref{pro:Equi-RWC:1} and \ref{pro:Equi-RWC:2} are equivalent.
    To show the equivalence of \ref{pro:Equi-RWC:1} and \ref{pro:Equi-RWC:3}, assume that the function $\varphi$ is $\rho$-weakly convex relative to $h$ on $S$; that is, for any $x_1,x_2\in S$ and any $\lambda\in[0,1]$
    \[
    \varphi(\lambda x_1+(1-\lambda)x_2) +\rho h(\lambda x_1+(1-\lambda)x_2) \leq \lambda \Big(\varphi(x_1)+\rho h(x_1)\Big) + (1-\lambda) \Big(\varphi(x_2) + \rho h(x_2)\Big).
    \]
    Rearranging terms, this can be equivalently written as
    \begin{align*}
        \varphi(\lambda x_1+(1-\lambda)x_2) \leq \lambda \varphi(x_1) &+ (1-\lambda) \varphi(x_2)\\
        &+ \rho \bigg( \lambda \Big(h(x_1)- h(\lambda x_1+(1-\lambda)x_2)\Big) + (1-\lambda) \Big(h(x_2) - h(\lambda x_1+(1-\lambda)x_2)\Big)\bigg),
    \end{align*}
    which ensures \cref{eq:weakconvexity}.
\end{proof}

\begin{cor}
    Let \(\func{h}{\R^n}{\Rinf}\) be a convex function, let $S\subseteq \interior \dom(h)$ be a convex nonempty set, and let $\rho> 0$. If the function \(\func{\varphi}{\R^n}{\Rinf}\) is $\rho$-weakly convex relative to $h$ on $S$, then, for any $x_1,x_2\in S$ and any $\lambda\in[0,1]$,
    \[
     \varphi(\lambda x_1+(1-\lambda)x_2) \leq \lambda \varphi(x_1) + (1-\lambda) \varphi(x_2) + \rho\min\{\lambda ,\theta_h^{-1}(1-\lambda)\} \D(x_1,x_2).
    \]
\end{cor}
\begin{proof}
    By \Cref{pro:Equi-RWC}, $\rho$-weak convexity of $\varphi$ relative to $h$ on $S$ implies that the function \(x\mapsto \varphi(x)+\rho\D(x,y)\) is convex on $S$ for any $y\in \interior \dom(h)$, i.e., for any $x_1,x_2 \in S$, any $\lambda\in[0,1]$, and any $y\in \interior \dom(h)$, one has
    \begin{equation*}
            \varphi(\lambda x_1+(1-\lambda)x_2) \leq \lambda \varphi(x_1) + (1-\lambda) \varphi(x_2)
            + \rho\Big(\lambda\D(x_1,y) +(1-\lambda)\D(x_2,y) -
            \D\big(\lambda x_1+(1-\lambda)x_2,y\big)\Big).
    \end{equation*}
    Let us consider the above inequality with two specific choices of $y$: first, $y=x_1$; second, $y=x_2$. Moreover, it is evident that
    \[
    \varphi(\lambda x_1+(1-\lambda)x_2) \leq \lambda \varphi(x_1) + (1-\lambda) \varphi(x_2)
            + \rho(1-\lambda)\D(x_2,x_1),
    \]
    \[
    \varphi(\lambda x_1+(1-\lambda)x_2) \leq \lambda \varphi(x_1) + (1-\lambda) \varphi(x_2)
            + \rho\lambda\D(x_1,x_2).
    \]
    Invoking the Bregman asymmetry bound $\D(x_2,x_1)\le \theta_h^{-1} \D(x_1,x_2)$, our desired inequality holds.
\end{proof}

We next investigate some calculus around the class of relatively weakly convex functions. Since the proofs of these results are straightforward consequences of \Cref{def:relSmooth}, we omit the details of the proofs.

\begin{prop}[{\bf Relative weak convexity calculus}]
\label{pro:relWeakConvCal}
    The following statements hold:
    \begin{enumerateq}
        \item
            Let $\func{h}{\R^n}{\Rinf}$ be a kernel function, let $S\subseteq \interior\dom(h)$ be a nonempty convex set, and let $\rho>0$.
            Let the function $\func{g}{\R^n}{\R}$ be $\rho$-weakly convex relative to $h$ on $S$.
            \begin{enumerateq}[label=$\bullet$]
                \item[(i)]
                    If $\func{\varphi}{\R^n}{\R}$ is a convex function on $S$, then $g+\varphi$ is $\rho$-weakly convex relative to $h$ on $S$;
                \item[(ii)]
                    Given $\varrho \in [\rho, \infty)$, the function $g$ is $\varrho$-weakly convex relative to $h$ on $S$;
                \item[(iii)]
                    Given $\varrho>0$, the function $\varrho g$ is $\varrho\rho$-weakly convex relative to $h$ on $S$.
            \end{enumerateq}
        \item 
            Let $\mathcal{I}$ be a finite index set,
            let $\func{h_i}{\R^n}{\Rinf},~i\in \mathcal{I},$ be a kernel function, let $S_{\mathcal{I}}\subseteq \cap_{i\in \mathcal{I}}\interior\dom(h_i)$ be a nonempty convex set, and $\rho_i>0,~i\in\mathcal{I}$.
            Let $\func{g_i}{\R^n}{\R}$, $i\in\mathcal{I}$, be $\rho_{i}$-weakly convex relative to $h_i$ on $S_{\mathcal{I}}$. 
        \begin{enumerateq}[label=$\bullet$]
        \item[(i)]
            $\sum_{i} g_i$ is $1$-weakly convex relative to $\sum_{i} \rho_ih_i$ on $S_{\mathcal{I}}$;
        \item[(ii)]
            If $h_{\mathcal{I}}=h_{i}$, $i\in \mathcal{I},$ then the function $ g_{\mathcal{I}}(x)=\max_{i} g_i(x)$ is $\rho_{\mathcal{I}}$-weakly convex relative to $h_{\mathcal{I}}$ on $S_{\mathcal{I}}$ at which $\rho_{\mathcal{I}}=\max_{i} \rho_i$.
    \end{enumerateq}
    \end{enumerateq}
\end{prop}

Local Lipschitz continuity and subdifferential regularity (see \cite[Definition 2.3.4]{clarke1990optimization}) are fundamental properties in nonsmooth and variational analysis. As in the convex setting, we show that relatively weakly convex functions also possess these properties. Consequently, the Fr\'{e}chet subdifferential is nonempty and coincides with the other standard subdifferentials in variational analysis \cite{rockafellar2011variational,clarke1990optimization,mordukhovich2006variational}.

\begin{thm}[{\bf Locally Lipschitz property and Subdifferential regularity}]\label{thm:LocLip-RWC}
    Let $\func{h}{\R^n}{\Rinf}$ be a kernel function, let $S\subseteq \interior\dom(h)$ be a nonempty convex set, and let $\rho>0$. If \(\func \varphi{\R^n}{\Rinf}\) is \(\rho\)-weakly convex relative to \(h\) on $S$, then it is locally Lipschitz continuous and subdifferentially regular on $\interior S$.
\end{thm}
\begin{proof}
    Suppose that $\varphi$ is a \(\rho\)-weakly convex function relative to the kernel \(h\) on $S$. Then, the function $\varphi+\rho h$ and $\rho h$ is convex on $S$. Furthermore, the function \(-\rho h\) is continuously differentiable on $S$. Thus, they both are locally Lipschitz continuous and subdifferentially regular (\cite[Proposition 2.3.6 (a) and (b)]{clarke1990optimization}) on $\interior S$. As a result, their addition, $\varphi=(\varphi+\rho h)+(-\rho h)$, is locally Lipschitz continuous and subdifferentially regular (\cite[Proposition 2.3.6 (c)]{clarke1990optimization}) on $\interior S$.
\end{proof}

In what follows, we present some characterizations of the relative weak convexity that are straightforward consequences of the definition.

\begin{prop}[\textbf{Characterizations II of relative weak convexity}]\label{pro:chara-RWC}%
    Let $\func{h}{\R^n}{\Rinf}$ be a kernel function, let $S\subseteq \interior\dom(h)$ be a nonempty convex and open set, and let $\rho>0$. Then, for a proper function \(\func \varphi{\R^n}{\Rinf}\), the following assertions are equivalent,
	\begin{enumerateq}[label=(\alph*), ref=(\alph*)]
	\item\label{pro:chara-RWC:1}
		\(\varphi\) is \(\rho\)-weakly convex relative to \(h\) on $S$;
	\item\label{pro:chara-RWC:2}%
		\(\varphi(y)\geq \varphi(x)+\innprod{\zeta}{y-x} - \rho\D_h(y,x)
		\)
		for any \(x,y\in S\) and any $\zeta \in \partial\varphi(x)$;
	\item\label{pro:chara-RWC:3}
		\(\innprod{\zeta-\eta}{x-y} \ge - \rho\innprod{\nabla h(x)-\nabla h(y)}{x-y}\)
		for any \(x,y\in S\), any $\zeta \in \partial\varphi(x)$, and any $\eta \in \partial\varphi(y)$;
	\item\label{pro:chara-RWC:4}
		\(\nabla^2\varphi(x) \succeq -\rho\nabla^2h(x)\),
		for any \(x\in S\), provided that \(\varphi,h\in\mathcal C^2(S)\).
	\end{enumerateq}
\end{prop}

We next establish the relative weak convexity for the composition, which is a generalization of the result studied in \cite[Lemma~4.2]{drusvyatskiy2019efficiency}.
We first define relatively smoothness.

\begin{defin}[\textbf{Relatively Bregman Jacobian smoothness}]\label{def:rbjs}
    Let $\func{h}{\R^n}{\Rinf}$ be a kernel, and let $S\subseteq \interior\dom(h)$ be a nonempty convex and open set. The mapping $\func{F}{\R^n}{\R^m}$ is said to be $L_F$-Bregman Jacobian smooth relative to $h$ on $S$ with $L_F>0$ if
    \begin{equation}\label{eq:relSmoothMap}
    \|F(x)-F(y)-J F(y)^T(x-y)\|\leq L_F \D(x,y), \quad \forall x, y\in S.
\end{equation}
\end{defin}

The following proposition provides a second-order characterization of relatively Bregman Jacobian smoothness.

\begin{prop}
    Let $\func{h}{\R^n}{\Rinf}$ and $\func{F}{\R^n}{\R^m}$ be twice continuously differentiable, and let $S\subseteq \interior\dom(h)$ be a nonempty convex and open set. If 
    \begin{equation*}
        \left\|\nabla^2 F(x)[y, y]\right\|_F \le L_F \nabla^2 h(x)[y,y],\quad\quad \forall x, y\in S,
    \end{equation*}
    then, the mapping $F$ is $L_F$-Bregman Jacobian smooth relative to $h$ on $S$.
\end{prop}
\begin{proof}
    For any $x, y \in S$, we apply the exact Taylor remainder formula in integral form:
    \begin{align*}
        \Big\|F(y) - F(x) - JF(x)[y-x]\Big\|_F 
        &= \left\| \int_0^1 (1-t) \nabla^2 F\big(x+t(y-x)\big)[y-x, y-x] dt \right\|_F \\
        &\le \int_0^1 (1-t) \Big\| \nabla^2 F\big(x+t(y-x)\big)[y-x, y-x] \Big\|_F dt \\
        &\le L_{F} \int_0^1 (1-t) \innprod{y-x}{\nabla^2 h\big(x+t(y-x)\big)[y-x]} dt \\
        &= L_{F} \D(y, x),
    \end{align*}
    verifying the result.
\end{proof}

The next proposition establishes the relative weak convexity for the composition function of the form $\psi\circ F$ at which $\func{\psi}{\R^m}{\Rinf}$ is a convex and Lipschitz continuous function and $\func{F}{\R^n}{\R^m}$ is a relatively Bregman Jacobian smooth mapping.

\begin{prop}[\textbf{Relative weak convexity of composition}]
\label{pro:composition-RWC}
    Let $\func{h}{\R^n}{\Rinf}$ be a kernel function and let $S\subseteq \interior\dom(h)$ be a nonempty convex and open set.
    Let $\func{\psi}{\R^m}{\Rinf}$ be a convex and Lipschitz continuous function with constant $L_\psi$, and let the mapping $\func{F}{\R^n}{\R^m}$ be $L_F$-Bregman Jacobian smooth relative to $h$ on $S$. Then, the composite function $\func{\varphi}{\R^n}{\Rinf}$ given by $\varphi:=\psi\circ F$ is $\rho$-weakly convex relative to $h$ on $S$ with $\rho=L_\psi L_F$. 
\end{prop}
\begin{proof}
    Let us fix points $x,y\in S$ and a subgradient $\zeta\in\partial \varphi(x)$, i.e., by the chain rule theorem, \cite[Theorem 2.3.10]{clarke1990optimization}, there exists $\eta\in\partial \psi(F(x))$ such that $\zeta=J F(x)^T\eta$. It follows from \cref{eq:relSmoothMap} and the convexity of $\varphi$ that
    $$
    \begin{array}{ll}
         \varphi(y)-\varphi(x) -\langle \zeta , y-x\rangle &= \psi(F(y)) -\psi(F(x)) - \langle JF(x)^T \eta , y-x\rangle \geq \langle \eta ,F(y)-F(x)\rangle - \langle \eta , JF(x)(y-x)\rangle\\[2mm]
         &\geq -\Vert \eta \Vert \,\Vert F(y) - F(x) - JF(x)(y-x)\Vert\geq -L_{\psi}L_{F} \D(x,y),
    \end{array}
    $$
    ensuring the relative weak convexity of $\varphi$ based on \Cref{pro:chara-RWC}~\ref{pro:chara-RWC:2}.
\end{proof}

\subsection{{\bf Optimization landscape of relatively weakly convex functions}}

The optimization landscape of relatively weakly convex functions exhibits a nuanced geometry, featuring local and global minima, maxima, and saddle points. Unlike convex functions, where any local minimum is global, relatively weakly convex functions may admit multiple local minima, complicating the identification of global optima. In addition, saddle points can significantly influence the dynamics of iterative algorithms, often slowing convergence or altering descent trajectories. Consequently, characterizing these critical points and the surrounding landscape is essential for developing effective algorithms in this setting.

In the following toy examples, we illustrate that such problems may exhibit saddle points or local maxima, while the global minimizers lie within a wide basin, supporting the intuition behind local convergence from favorable initializations.

\begin{exa}\label{exa:example_noniso_min}
    Let us consider the function $f:\mathbb{R}^2\to\mathbb{R}$ defined by
    \[
    f(x)=
    \begin{cases}
    4-4\|x\|,\hspace*{1cm} & 0\le \|x\|\le 1,\\[1mm]
    4\|x\|-4, & 1\le \|x\|\le 1.5,\\[1mm]
    2, & 1.5\le \|x\|\le 2,\\[1mm]
    4\|x\|-6, & \|x\|\ge 2,
    \end{cases}
    \]
    for $x\in \R^2$; see
    \Cref{fig:rwc_example}.
    Observe that the function is constantly equal to $2$ on the set $\{x\in \R^2 : 1.5\le \|x\|\le 2\}$. The origin $x=(0,0)$ and every point satisfying
    \(
    1.5\le \|x\| <2
    \) are local maximizers, whereas every point satisfying
    \(1.5<\|x\|\le 2
    \)
    is a local minimizer.
    Moreover, every point satisfying
    \(
    \|x\|=1
    \)
    is a global minimizer, showing that the set of global minimizers is nonisolated. In particular, the function is nonconvex. On the other hand, the function admits the representation
    \[
    f(x)
    =
    4-4\|x\|
    +4(\|x\|-2)_+
    -4(\|x\|-1.5)_+
    +8(\|x\|-1)_+,
    \]
    where
    \(
    (t)_+:=\max\{t,0\},
    \)
    following that $f$ is $4$-weakly convex relative to the convex function 
    \[
    h(x):=\|x\|+(\|x\|-1.5)_+.
    \]

    \begin{figure}[htp]
        \centering
        \subfloat{\includegraphics[width=7.5cm]{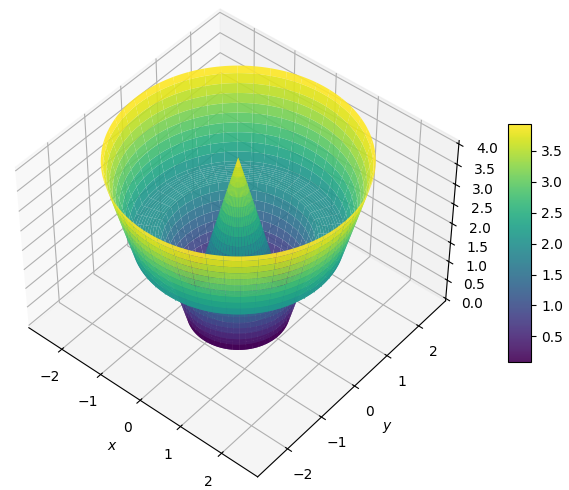}}%
        \qquad
        \subfloat{\includegraphics[width=7cm]{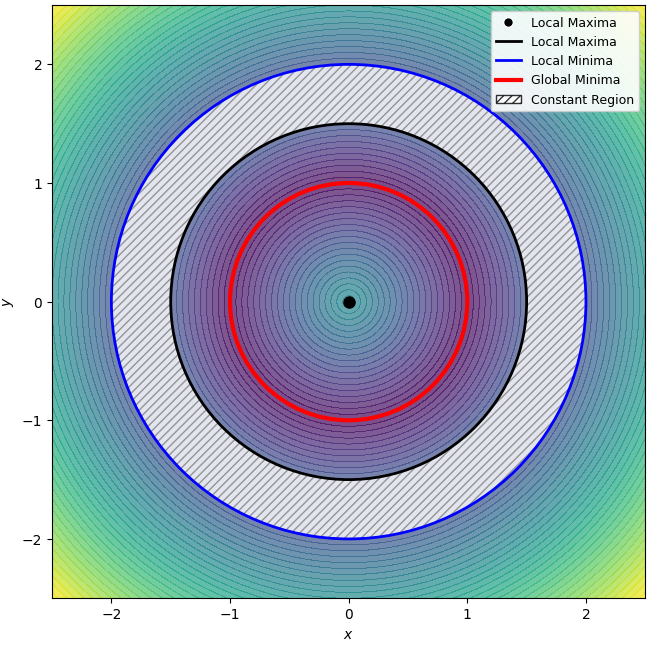}}%
        \caption{The annulus
        \(
        \{x\in\mathbb{R}^2:1.5\le\|x\|\le2\}
        \)
        (white region) is a flat stationary region where the function is constant, the origin
        \(
        (0,0)
        \)
        (black point) and the circle
        \(
        \{x\in \mathbb{R}^2:\|x\|=1.5\}
        \)
        (black circle) are the set of local maximizers, the circle
        \(
        \{x\in \mathbb{R}^2:\|x\|=2\}
        \)
        (blue circle) is the set of local maximizers, and the circle
        \(
        \{x\in \mathbb{R}^2:\|x\|=1\}
        \)
        (red circle) is the set of global minimizers.}\label{fig:rwc_example}
    \end{figure}
    \qed
\end{exa}

For many instances of the problem \cref{problem}, the global minimizers are surrounded by a broad basin, as illustrated in \Cref{fig:rwc_example,fig:weakconvExa1}. Consequently, if a local optimization method is initialized within this basin, it is likely to converge rapidly to a global minimizer. This observation motivates the design of two-stage optimization frameworks, where (i) the outer stage employs a coarse method, such as a spectral technique, to identify an initial point sufficiently close to the global minimizer, and (ii) the inner stage utilizes a local refinement algorithm to ensure fast convergence.

More specifically, under certain sufficient conditions, spectral methods, see, e.g., \cite[Section VIII]{chi2019nonconvex}, can reliably produce a point in the attraction region of a global solution. Once such a point is found, a suitable optimization routine, such as a projected subgradient method (cf. \Cref{sec:subGradMethod}), can be applied to drive the iterates toward the global minimizer.

The next result characterizes a neighborhood of the solution set to problem \cref{problem} that is free of spurious stationary points. This region is guaranteed by the combination of relative weak convexity and a sharpness error bound, and it provides a theoretically sound initialization region. That is, any algorithm started within this neighborhood is assured to generate a sequence that converges to a global minimizer.

\begin{ass}\label{ass2}
    We assume that the Bregman distance $D_h$ admits the following local Bregman growth condition, i.e., there exist constants $\nu\in (0,1]$ and $G_h>0$ such that
        \begin{align}\label{ineq:BGC}
            \inf_{y\in\proj_{\X^*}(x)}\D(y,x) \le G_{h} \dist^{1+\nu}(x;\mathcal{X}^*), \quad\quad\forall x\in \X, \,\, \dist(x;\mathcal{X}^*)< \big(\tfrac{\mu}{\rho G_h}\big)^{\nicefrac{1}{\nu}}.
        \end{align}
\end{ass}

It follows from \Cref{fact:BD} that if the function \(h\) has \(\nu\)-H\"older continuous gradient around the minimizer set, then the above assumption is automatically valid.

Throughout, we assume that \Cref{ass:basic,ass2} are in effect.
For any $\beta \in (0,1]$, let us define the tube
\begin{equation}\label{eq:Tbeta2}
    \mathcal{T}_\beta :=\left\{x \in  \X: \dist(x;\mathcal{X}^*) < \Big(\tfrac{\beta\mu}{\rho G_{h}}\Big)^{\nicefrac{1}{\nu}} \right\},
\end{equation}
which contains no extraneous stationary points of the problem, due to \cref{pro:dis2}.

\begin{prop}\label{pro:dis2}
    Let $x\in \T_1$ be a stationary point of \cref{problem}. Then $x\in \mathcal{X}^*$. 
\end{prop}
\begin{proof}
    By indirect proof, assume that there exists $x\in \T_1\setminus \mathcal{X}^*$ which is a stationary point of \cref{problem}, i.e., $0\in \partial(f+\delta_X)(x)$. Then, from $\rho$-weak convexity of $f$ relative to $h$, \cref{pro:chara-RWC} yields that
    $$f(y) - f(x) \geq -\rho \D(y,x), \quad \forall y\in X.$$
    Setting $x^* \in \proj_{\mathcal{X}^*}{(x)}$ such that $\D(x^*, x) = \inf_{y\in \proj_{\X^*}(x)} \D(y,x)$ and applying sharpness error bound together with local Bregman growth condition \cref{ineq:BGC}, we come to
    \begin{equation*}
             \mu \dist(x;\mathcal{X}^*)\leq f(x) - f(x^*) \leq \rho \D(x^*,x) \le \rho G_{h} \dist^{1+\nu}(x;\mathcal{X}^*),
    \end{equation*}
    making a contradiction with $x\in \T_1$. 
\end{proof}

Additionally, we set
$$L := \sup \bigg\{\Vert \zeta\Vert : \zeta\in \partial f(x),~ x \in \mathcal{T}_{1}\bigg\}.$$
The following lemma provides a key relationship between $\mu$ and $L$. Let us define $\tau:=\tfrac{\mu}{L}$.

\begin{lem}\label{lem:tau}
    It holds that $\tau=\tfrac{\mu}{L} \in (0,1]$.
\end{lem}

\begin{proof}
    Let us consider
    $x\in \mathcal{T}_{1} \setminus \mathcal{X}^*$.
    For $x^* \in \proj_{\mathcal{X}^*} {(x)}$, invoking the mean value theorem \cite[Theorem 3.51]{mordukhovich2006variational}, there exist $z\in [x,\bar x)$ and $\zeta \in \partial f(z)$ satisfying
    $f(x) - f(x^*) \leq \langle \zeta , x- x^* \rangle.$
    Applying the sharpness error bound property, we come to
    $$\mu \dist(x;\mathcal{X}^*)\leq f(x) -f^* \leq \Vert\zeta\Vert ~\Vert x-x^*\Vert\leq L\dist(x;\mathcal{X}^*),$$
    confirming $\tau\in (0,1]$.
\end{proof}


\section{Projected subgradient algorithm}\label{sec:subGradMethod}

In this section, we present projected subgradient methods for the nonsmooth and constrained relatively weakly convex optimization problems of the form \cref{problem} and establish their convergence analysis for several choices of step-sizes.

\vspace{3mm}
\RestyleAlgo{boxruled}
\begin{algorithm}[H]
\DontPrintSemicolon
\KwIn{$x_0\in \mathcal{T}_\beta$,~ $\beta \in (0,1)$,~ $\alpha_0>0$;}
\Begin
{   
    \While{the stopping criterion does not hold}{
       Choose $\zeta_k \in \partial f(x_k)$;\;
       
       Set $x_{k+1}=\proj_X\left(x_k-\alpha_k\tfrac{\zeta_k}{\|\zeta_k\|}\right)$ and $k=k+1$;\;
       
        }
    Set $x_{best}=x_k$;\;
}
\KwOut{$x_{best}$.} 
\caption{PSGA (Projected Subgradient Algorithm)\label{alg:PSGA}}
\end{algorithm}

\vspace{3mm}

In \Cref{alg:PSGA}, the step-size $\alpha_{k}>0$ plays a key role in the algorithm's progression.
In fact, by employing different step-sizes we can define various projected subgradient methods, where they can then be compared in terms of the convergence rates and overall numerical performance. In this paper, we investigate several commonly step-size strategies for PSGA, namely:
\begin{enumerateq}
\item \emph{\textbf{Constant step-size}} (see \Cref{sec:PSGA-C}):
\[
\alpha_k=\alpha>0,\qquad \forall\,k.
\]

\item \emph{\textbf{Nonsummable diminishing step-size (ND)}} (see \Cref{sec:PSGA-DS}):
\[
\alpha_k\ge0,\qquad
\lim_{k\to\infty}\alpha_k=0,\qquad
\sum_{k=0}^{\infty}\alpha_k=\infty.
\]
A commonly used example is
\[
\alpha_k=\frac{\lambda}{(k+k_0)^r},
\qquad
\lambda,k_0>0,\;\;0<r<1.
\]
The classical harmonic step-size, corresponding to $r=1$, also belongs to this class.

\item \emph{\textbf{Square-summable yet nonsummable step-size (SSN)}} (see \Cref{sec:PSGA-DS}):
\[
\alpha_k\ge0,\qquad
\sum_{k=0}^{\infty}\alpha_k=\infty,\qquad
\sum_{k=0}^{\infty}\alpha_k^2<\infty.
\]
A standard example is
\[
\alpha_k=\frac{\lambda}{(k+k_0)^r},
\qquad
\lambda,k_0>0,\;\;\frac12<r\le1.
\]

\item \emph{\textbf{Geometrically decaying step-size (GD)}} (see \Cref{sec:PSGA-DS}):
\[
\alpha_k=\lambda q^k,
\qquad
\lambda>0,\;\;0<q<1,
\]
which is a classical geometrically decaying rule.
\end{enumerateq}

The convergence properties of PSGA under these step-size rules are investigated in the subsequent subsections.

We begin with the subsequent lemma establishing a fundamental recurrence, providing an upper bound for the update step based on the chosen step-size. This result is instrumental in deriving the convergence rate of the projected subgradient methods.

\begin{lem}[{\bf Basic inequalities I}]\label{lem:Basic1}
    Let the sequence \(\seq{x_k}\) be generated by PSGA, the following statements hold:
    \begin{enumerateq}[label=(\alph*), ref=(\alph*)]
        \item\label{lem:Basic1:1}
        If \(x_k \in \mathcal{T}_\beta\), one has
        \begin{align}
            \dist^2(x_{k+1};\mathcal{X}^*) \leq \dist^2(x_{k};\mathcal{X}^*)-\tfrac{2(1-\beta)\alpha_{k}}{\Vert \zeta_{k} \Vert} \left(f(x_{k})-f^{*}\right) +\alpha_k^2,\label{eq:lem-Basic01}
        \end{align}
        and the maximum decrease is given for $\alpha_k=(1-\beta)\tfrac{f(x_k)-f^*}{\|\zeta_k\|}$.
        Moreover,
        \begin{align}
            \dist^2(x_{k+1};\mathcal{X}^*) \leq \dist^2(x_{k};\mathcal{X}^*) -\tfrac{2(1-\beta)\mu\alpha_{k}}{\Vert \zeta_{k} \Vert} \dist(x_{k};\mathcal{X}^*) + \alpha_k^2;\label{eq:lem-Basic02}
        \end{align}
        \item\label{lem:Basic1:2}
        If the step-sizes satisfy
        \(0<\alpha_k \le \min\big\{1,2(1-\beta)\tau\big\} \Big(\tfrac{\beta\mu}{\rho G_{h}}\Big)^{\tfrac{1}{\nu}},\)
        then \(\seq{x_{k}}\subseteq \mathcal{T}_\beta\).
    \end{enumerateq}
\end{lem}
\begin{proof}
    \ref{lem:Basic1:1} From \Cref{alg:PSGA}, if $\zeta_k=0$, the claims are evident then from $x_{k+1}=x_{k}\in \mathcal{X}^*$.
    Without loss of generality, we assume $\zeta_k\neq 0$ and $x_k\not\in \X^*$. Let us consider $x^* \in \proj_{\mathcal{X}^*}(x_{k})$ such that $\D(x^*, x_{k}) = \displaystyle\inf_{y\in \proj_{\X^*}(x_{k})} \D(y,x_{k})$.
    Using the nonexpansiveness of $\operatorname{proj}_{X}$ and applying $\rho$-weak convexity of $f$ relative to $h$, we obtain
    \begin{align}
        \dist^2(x_{k+1};\mathcal{X}^*) &\leq \Vert x_{k+1}- x^* \Vert^{2} \leq \Vert(x_{k}-x^*)-\alpha_{k} \tfrac{\zeta_{k}}{\Vert \zeta_{k} \Vert}\Vert^{2} = \Vert x_{k}-x^* \Vert^{2}+\tfrac{2 \alpha_{k}}{\Vert\zeta_{k}\Vert} \langle\zeta_{k}, x^* - x_{k}\rangle + \alpha_{k}^{2}\nonumber\\
        &\le \dist^2(x_{k};\mathcal{X}^*)+\tfrac{2 \alpha_{k}}{\Vert \zeta_{k} \Vert} \Big(f^{*}-f(x_{k})+ \rho \D(x^*,x_k)\Big) + \alpha_{k}^{2}.\label{eq:lem-Basic05}
    \end{align}
    Inasmuch as $x_{k}\in \mathcal{T}_{\beta}$, local Bregman growth condition \cref{ineq:BGC} together with sharpness error bound property yield that
    \begin{equation}\label{eq:lem-Basic03}
        \rho \D(x^*,x_{k})=\rho\displaystyle\inf_{y\in \proj_{\X^*}(x_{k})} \D(y,x_{k}) \le \rho G_{h} \dist^{1+\nu}(x_{k};\mathcal{X}^*) < \beta\mu \dist(x_{k};\mathcal{X}^*) \leq \beta(f(x_{k})-f^{*}).
    \end{equation}
    Substituting \cref{eq:lem-Basic03} into \cref{eq:lem-Basic05} verifies inequality \cref{eq:lem-Basic01}.
    Furthermore, the function
    $$\alpha\mapsto -\tfrac{2(1-\beta)\alpha}{\Vert \zeta_{k} \Vert} \left(f(x_{k})-f^{*}\right) +\alpha^2,$$
    is a convex function on $(0,+\infty)$ and attains its minimum at $\alpha_k=(1-\beta)\tfrac{f(x_k)-f^*}{\|\zeta_k\|}$ which leads to the maximum decrease in \cref{eq:lem-Basic01}.
    Moreover, the argument for \cref{eq:lem-Basic02} is derived from \cref{eq:lem-Basic01}, relying on a sharpness error bound.\\
    \ref{lem:Basic1:2} It follows from \Cref{alg:PSGA} that $x_0\in\T_\beta$. By induction, we assume that $x_{k}\in \T_\beta$ and prove that $x_{k+1}\in \T_\beta$. Let us consider two cases: $(i)$ $\dist(x_{k};\mathcal{X}^*)\ge   \tfrac{\alpha_k\Vert \zeta_{k} \Vert}{2(1-\beta)\mu}$; $(ii)$ $\dist(x_{k};\mathcal{X}^*)<   \tfrac{\alpha_k\Vert \zeta_{k} \Vert}{2(1-\beta)\mu}$.
    
    In Case~(i), the inequality \cref{eq:lem-Basic02} implies
    $$
    \dist^2(x_{k+1};\mathcal{X}^*) \leq \dist^2(x_{k};\mathcal{X}^*)
    < \Big(\tfrac{\beta\mu}{\rho G_{h}}\Big)^{\tfrac{2}{\nu}},
    $$
    i.e., $x_{k+1}\in \mathcal{T}_\beta$. 
    
    Let us consider Case~(ii).
    From the convexity of the function
    $t\mapsto t^2 - \tfrac{2(1-\beta)\mu\alpha_{k}}{\Vert \zeta_{k} \Vert} t   + \alpha_{k}^{2}$
    on $[0,\tfrac{\alpha_{k}\Vert \zeta_{k} \Vert}{2(1-\beta)\mu}]$
    achieving its the maximum at $t=0$ or $t=\tfrac{\alpha_{k}\Vert \zeta_{k} \Vert}{2(1-\beta)\mu}$, inequality \cref{eq:lem-Basic02} yields that
    \begin{align*}
        \dist^2(x_{k+1};\mathcal{X}^*) &\le \max \left\{ \alpha^{2}_{k},~
        \tfrac{\alpha_{k}^2\Vert \zeta_{k} \Vert^2}{4(1-\beta)^2\mu^2} \right \} \leq \alpha_k^2 \max \left\{ 1,~
        \tfrac{1}{4(1-\beta)^2\tau^2} \right\}\\
        &\le \max \left\{ 1,~
        \tfrac{1}{4(1-\beta)^2\tau^2} \right\}\min\Big\{1,4(1-\beta)^2\tau^2\Big\} \Big(\tfrac{\beta\mu}{\rho G_{h}}\Big)^{\tfrac{2}{\nu}}
        =\Big(\tfrac{\beta\mu}{\rho G_{h}}\Big)^{\tfrac{2}{\nu}},
    \end{align*}
    where the second and third inequalities are derived from $\|\zeta_k\|\le L$ and the bound on $\alpha_{k}$, respectively. Therefore, $x_{k+1}\in \mathcal{T}_\beta$ and the proof is completed.
\end{proof}

In the following lemma, we provide an upper bound for the function gap $f^*_k - f^*$ where $f_{k}^{*}:=\min\{f(x_{i}) : i=0,1,\ldots,k\}$.

\begin{lem}[{\bf Basic inequalities II}]\label{lem:Basic2}
    Let the sequence $\seq{x_{k}}\subseteq \mathcal{T}_\beta$ be generated by PSGA. Then,
    \begin{align}
        f(x_{k}) - f^{*} \leq \tfrac{L\dist^2(x_{k};\mathcal{X}^*)+ L \alpha_{k}^{2}}{2(1-\beta)\alpha_{k}},\quad\quad \forall k\ge 0,~\,\label{eq:lem:basic2:01}\\
        f_{k}^{*}- f^{*} \leq \tfrac{\big(f(x_{0}) - f^{*}\big)^2+ \mu^2 \sum_{i=0}^{k}\alpha_{i}^{2}}{2(1-\beta)\mu\tau\sum_{i=0}^{k}\alpha_{i}},\quad\quad \forall k\ge 0.\label{eq:lem:basic2:02}
    \end{align}
\end{lem}
\begin{proof}
    Inasmuch as $\|\zeta_k\|\le L$, inequality \cref{eq:lem-Basic01} simplifies to
    \begin{align*}
            \tfrac{2(1-\beta)\alpha_{k}}{L} \left(f(x_{k})-f^{*}\right) \leq \dist^2(x_{k};\mathcal{X}^*)- \dist^2(x_{k+1};\mathcal{X}^*)+\alpha_k^2\le \dist^2(x_{k};\mathcal{X}^*)+\alpha_k^2,
        \end{align*}
    which directly yields \cref{eq:lem:basic2:01}.
    On the other hand,
    Taking the sum of both sides of the first above inequality from $i=0$ to $i=k$, we obtain
    \begin{equation*}
        \sum_{i=0}^{k}\tfrac{2(1-\beta)\alpha_{i}}{L} \big(f(x_{i}) - f^{*}\big)
        \leq \dist^2(x_{0};\mathcal{X}^*) - \dist^2(x_{k+1};\mathcal{X}^*) + \sum_{i=0}^{k} \alpha_{i}^{2},
    \end{equation*}
    leading to
    \begin{align}\label{eq:lem-Basic01L}
       \tfrac{2(1-\beta)\big(f^{*}_{k} - f^{*}\big)}{L}\sum_{i=0}^{k}\alpha_{i}\le \dist^2(x_{0};\mathcal{X}^*)+ \sum_{i=0}^{k}\alpha_{i}^{2}\le \tfrac{\big(f(x_0)-f^*\big)^2}{\mu^2}
        + \sum_{i=0}^{k}\alpha_{i}^{2},
    \end{align}
    ensuring \cref{eq:lem:basic2:02}.
\end{proof}

\begin{lem}[\textbf{Convergence analysis}]\label{lem:ConvAna}
    Let the sequence $\seq{x_{k}}$ be generated by PSGA. If $\displaystyle\lim_{k\to\infty}\dist(x_k;\X^*)=0$, the following statements hold:
    \begin{enumerateq}[label=(\alph*), ref=(\alph*)]
        \item\label{lem:ConvAna:1}
            $f(x_k) - f^* \leq L\dist(x_k;\mathcal{X}^*)$ for all $x_k\in\T_1$, and $\displaystyle\lim_{k\to\infty}f(x_k)=f^*$;
        \item\label{lem:ConvAna:2}
            All cluster points of the sequence $\seq{x_k}$ are global optimal solutions, if any;
        \item\label{lem:ConvAna:3}
            If $\sum_{k=0}^{\infty} \alpha_k<\infty$, then $\seq{x_k}$ converges to a global optimal solution.
    \end{enumerateq}
\end{lem}
\begin{proof}
    \ref{lem:ConvAna:1} It follows from $\displaystyle\lim_{k\to\infty}\dist(x_k;\X^*)=0$ that there exists $k_0\in \N_0$ such that $x_k\in \T_1$ for all $k\ge k_0$ implying $\|\zeta_k\|\le L$. Let us consider $x^*_k \in \proj_{\mathcal{X}^*} {(x_k)}$. Using the mean value theorem \cite[Theorem 3.51]{mordukhovich2006variational}, there exist $z_k\in [x_k,x^*_k)$ and $\zeta_k \in \partial f(z_k)$ satisfying
    $f(x_k) - f^* \leq \langle \zeta_k , x_k- x_k^* \rangle.$ It can be concluded that
    $$0\le f(x_k) - f^* \leq \Vert\zeta_k\Vert ~\Vert x_k- x_k^*\Vert\leq L\dist(x_k;\mathcal{X}^*),\quad\quad k\ge k_0,$$
    confirming $f(x_k)\to f^*$.\\
    \ref{lem:ConvAna:2} Assume that $\bar x\in X$ is a cluster point of $\seq{x_k}$, i.e., there exists a subsequence $\seqj{x_{k_j}}$ of $\seq{x_{k}}$ such that $x_{k_j}\to \bar x$ as $j\to \infty$. Then
    \[
    \dist(\bar x;\X^*) = \lim_{j\to \infty} \dist(x_{k_j};\X^*) = \lim_{k\to \infty} \dist(x_{k};\X^*) =0,
    \]
    which verifies that $\bar x\in \X^*$ due to the closeness of $\X^*$.\\
    \ref{lem:ConvAna:3} From $x_{k+1}=\proj_X\left(x_k-\alpha_k\tfrac{\zeta_k}{\|\zeta_k\|}\right)$ and the nonexpansiveness of $\operatorname{proj}_{X}$, it holds that
    $\|x_{k+1}-x_{k}\|\le \alpha_k$.
    Additionally,
    $\sum_{k=0}^{\infty} \|x_{k+1}-x_k\| \le \sum_{k=0}^{\infty} \alpha_k < \infty$
    implies the convergence of $\seq{x_k}$ to a global optimal solution based on Assertion~\ref{lem:ConvAna:2}, giving our desired result.
\end{proof}

\subsection{{\bf PSGA with constant step-size}}\label{sec:PSGA-C}
In this subsection, we consider a projected subgradient method with a constant step-size, $\alpha_{k}=\alpha>0$ for all $k$.
We demonstrate that, with a suitable choice of initialization, the distance sequence $\seq{\dist(x_k; \mathcal{X}^*)}$ exhibits linear convergence up to a predetermined threshold.
\begin{thm}[{\bf Convergence rate of constant PSGA}]\label{thm:PSGA-C1}
    Let the sequence $\seq{x_{k}}$ be generated by PSGA with the constant step-size $\alpha_{k}=\alpha$ satisfying
    $
    0<\alpha\le\tfrac{2(1-\beta)\tau}{\sqrt{1+4(1-\beta)^2\tau^{2}}}\Big(\tfrac{\beta\mu}{\rho G_{h}}\Big)^{\nicefrac{1}{\nu}}
    $. Then,
    \begin{equation}\label{eq:Ineq-CS}
        \dist^2(x_k; \mathcal{X}^*)-D_*^2 \leq \max \left\{\left(\dist^2(x_0; \mathcal{X}^*)-D_*^2\right)q^{k},
        \alpha^{2}\right\},   \quad\quad \forall k\ge 0,
    \end{equation}
    where $D_* := \tfrac{\alpha}{2(1-\beta)\tau}$ and $q:= 1- (1-\beta)\alpha\tau\Big(\tfrac{\rho G_{h}}{\beta\mu}\Big)^{\nicefrac{1}{\nu}}>0.$
\end{thm}
\begin{proof}
    By \Cref{lem:tau}, $\tau \in (0,1]$, implying $\tfrac{2(1-\beta)\tau}{\sqrt{1+4(1-\beta)^2\tau^{2}}} < 1<\tfrac{1}{(1-\beta)\tau}$. Thus, $q>0$, due to the upper bound of $\alpha$.
    Moreover, \(\tfrac{2(1-\beta)\tau}{\sqrt{1+4(1-\beta)^2\tau^{2}}} < \min\{1,2(1-\beta)\tau\}\), ensuring $\seq{x_k}\subseteq \T_\beta$ by \Cref{lem:Basic1}~\ref{lem:Basic1:2}.
    We now verify the inequality \Cref{eq:Ineq-CS} using induction. At $k=0$, the inequality holds trivially. Assume that for some $k \in \mathbb{N}$, the following holds (inductive hypothesis):
    $$
    \dist^2(x_i; \mathcal{X}^*)-D_*^2 \leq \max \left\{q^{i}\left(\dist^2(x_0; \mathcal{X}^*)-D_*^2\right),
        \alpha^{2}\right\}, \quad \forall i=0,1, \ldots, k.
    $$
    We prove the claim for $(k+1)-$th step. It follows from \cref{lem:Basic1}~\ref{lem:Basic1:1}, $\|\zeta_k\|\le L$, and the concavity of function $t\mapsto \sqrt t$ on $[0,+\infty)$ that
    \begin{align*}
        \dist^2(x_{k+1};\mathcal{X}^*) &\leq \dist^2(x_{k};\mathcal{X}^*) -2(1-\beta)\tau\alpha \dist(x_{k};\mathcal{X}^*) + \alpha^2\\
        &=\dist^2(x_{k};\mathcal{X}^*) +2(1-\beta)\tau\alpha\Big( D_*-\dist(x_{k};\mathcal{X}^*) \Big)\\
        &\leq \dist^2(x_{k};\mathcal{X}^*) + \tfrac{(1-\beta)\tau\alpha}{\dist(x_{k};\mathcal{X}^*)}\Big(D_*^2 -\dist^2(x_{k};\mathcal{X}^*)\Big),
    \end{align*}
    i.e.,
    \begin{equation*}
        \dist^2(x_{k+1};\mathcal{X}^*)-D_{*}^2 \leq \bigg(1 - \tfrac{(1-\beta)\tau\alpha}{\dist(x_{k};\mathcal{X}^*)}\bigg)\Big(\dist^2(x_{k};\mathcal{X}^*)-D_{*}^2\Big).
    \end{equation*}
    There are two possible cases: $(i)$ $\dist(x_{k};\mathcal{X}^*) \geq D_{*}$; $(ii)$ $\dist(x_{k};\mathcal{X}^*) < D_{*}$.
    
    In Case~$(i)$, it follows from  $x_k\in \T_\beta$ that $\dist^2(x_{k+1};\mathcal{X}^*)-D_{*}^2 \leq q \Big(\dist^2(x_{k};\mathcal{X}^*)-D_{*}^2\Big)$.
    In Case~$(ii)$, relying once again on \cref{lem:Basic1}~\ref{lem:Basic1:1}, we obtain
    $$
    \dist^2(x_{k+1};\mathcal{X}^*)-D_{*}^2 \leq \dist^2(x_{k};\mathcal{X}^*) - D_{*}^2 -2(1-\beta)\tau\alpha \dist(x_{k};\mathcal{X}^*) + \alpha^2< \alpha^{2}.
    $$
    Hence, in both cases and by the induction assumption, it can be deduced
    $$
        \dist^2(x_{k+1};\mathcal{X}^*)-D_{*}^2 \leq \max \left\{q\left(\dist^2(x_{k};\mathcal{X}^*)-D_{*}^2\right), \alpha^{2}\right\}\leq \max \left\{q^{k+1}\left(\dist^2(x_{0};\mathcal{X}^*)-D_{*}^2\right),
        \alpha^{2}\right\},
    $$
    completing the proof.
\end{proof}

\Cref{thm:PSGA-C1} guarantees that the sequence $\seq{\dist(x_k; \mathcal{X}^*)}$ decreases linearly to a value below $D_{*}$.  Moreover, once $\dist(x_k; \mathcal{X}^*)$ falls below $D_{*}$, it stays within a neighborhood of this threshold.

The next theorem provides an upper bound for the function gap sequence.

\begin{thm}[{\bf Convergence rate of constant PSGA}]\label{thm:Conv-C2}
    Let the sequence $\seq{x_{k}}$ be generated by PSGA with the constant step-size $\alpha_{k}=\alpha$ satisfying
    \(0<\alpha \le \min\big\{1,2(1-\beta)\tau\big\} \Big(\tfrac{\beta\mu}{\rho G_{h}}\Big)^{\nicefrac{1}{\nu}}\). Then,
    $$
    0\leq f^{*}_{k} - f^{*} \leq \tfrac{L\alpha}{1-\beta}, \quad\quad\forall k\geq \tfrac{1}{\alpha^{2}}\Big(\tfrac{\beta\mu}{\rho G_{h}}\Big)^{\nicefrac{2}{\nu}}-1.
    $$
\end{thm}
\begin{proof}
    By \cref{lem:Basic1}~\ref{lem:Basic1:2}, the upper bound of $\alpha$ ensures that $\seq{x_{k}}\subseteq \mathcal{T}_\beta$. Furthermore, from \Cref{lem:Basic2} and inequality \cref{eq:lem-Basic01L}, we deduce
    $$
    f_{k}^{*}- f^{*} \leq \tfrac{L\dist^2(x_0;\X^*)+ L \sum_{i=0}^{k}\alpha_{i}^{2}}{2(1-\beta)\sum_{i=0}^{k}\alpha_{i}}
    \leq  \tfrac{L\Big(\tfrac{\beta\mu}{\rho G_{h}}\Big)^{\nicefrac{2}{\nu}}+ L \alpha^{2}(k+1)}{2(1-\beta)\alpha (k+1)} \leq \tfrac{2L\alpha^{2}(k+1)}{2(1-\beta)\alpha (k+1)}=\tfrac{L\alpha}{1-\beta}, \quad\quad\forall k\geq \tfrac{1}{\alpha^{2}}\Big(\tfrac{\beta\mu}{\rho G_{h}}\Big)^{\nicefrac{2}{\nu}}-1,
    $$
    adjusting our claim.
\end{proof}

\subsection{{\bf PSGA with diminishing step-sizes}}\label{sec:PSGA-DS}

In this section, we analyze the projected subgradient method with a diminishing step-size aimed at achieving convergence to an optimal solution. We begin by examining the method's behavior under two step-size regimes: $(i)$ nonsummable diminishing step-sizes (ND), and $(ii)$ square-summable yet nonsummable step-sizes (SSN). The following theorem establishes the convergence of the sequence of gap values $\seq{f_k^* - f^*}$, where $\seq{x_k}$ is generated by PSGA using ND step-size satisfying
$$
\alpha_{k}\geq 0,\quad \lim_{k\to\infty}\alpha_{k}=0, \quad \sum_{k=1}^{\infty}\alpha_{k} =\infty.
$$
In addition, it ensures the convergence of a subsequence of the iterates $\seq{x_k}$.

\begin{thm}[{\bf Convergence rate of ND PSGA}]\label{thm:Conv-ND}
    Let the sequence $\seq{x_{k}}$ be generated by PSGA with ND step-size $\alpha_{k}$ satisfying
    \(0<\alpha_k \le \min\big\{1,2(1-\beta)\tau\big\} \Big(\tfrac{\beta\mu}{\rho G_{h}}\Big)^{\nicefrac{1}{\nu}}\). Then, the following statements hold:
    \begin{enumerateq}[label=(\alph*), ref=(\alph*)]
        \item\label{thm:Conv-ND:1}
        $\displaystyle\lim_{k\to \infty} f^{*}_{k}=\liminf_{k\to\infty} f(x_{k}) = f^{*}$;
        \item\label{thm:Conv-ND:2}
        $\displaystyle\liminf_{k\to\infty}\dist(x_{k};\X^*)= 0$;
        \item\label{thm:Conv-ND:3} If $\seq{x_{k}}$ is a bounded sequence, it has a convergent subsequence to some optimal solution. 
    \end{enumerateq}
\end{thm}
\begin{proof}
    By \cref{lem:Basic1}~\ref{lem:Basic1:2} and \cref{lem:Basic2}, $\seq{x_{k}}\subseteq \mathcal{T}_\beta$ and
    $$
    f_{k}^{*}- f^{*} \leq \tfrac{\big(f(x_{0}) - f^{*}\big)^2+ \mu^2 \sum_{i=0}^{k}\alpha_{i}^{2}}{2(1-\beta)\mu\tau\sum_{i=0}^{k}\alpha_{i}},\quad\quad \forall k\ge 0.
    $$
    Thus, the proof follow the same arguments as in the proof of \cite[Theorem 4.6]{rahimi2024projected}.
\end{proof}

The previous theorem established that the PSGA with a ND step-size guarantees subsequential convergence. We now extend this result by proving full convergence of the sequences $\seq{\dist(x_k;\X^*)}$ and $\seq{f(x_k)-f^*}$ at which $\seq{x_k}$ is generated by PSGA under a more restrictive step-size rule. Specifically, we assume the SSN step-sizes satisfying:
$$
\alpha_{k}\geq 0, \quad \sum_{k=1}^{\infty}\alpha_{k} =\infty, \quad \sum_{k=1}^{\infty}\alpha_{k}^{2} <\infty.
$$

\begin{thm}[{\bf Convergence rate of SSN PSGA}]\label{thm:Conv-SSN}
    Let the sequence $\seq{x_{k}}$ be generated by PSGA with SSN step-size $\alpha_{k}$ satisfying
    \(0<\alpha_k \le \min\big\{1,2(1-\beta)\tau\big\} \Big(\tfrac{\beta\mu}{\rho G_{h}}\Big)^{\nicefrac{1}{\nu}}\). Then, $\displaystyle\lim_{k\to\infty}\dist(x_{k};\X^*)= 0$, $\displaystyle\lim_{k\to\infty}f(x_k)=f^*$, and all cluster points of the sequence $\seq{x_k}$ are global optimal solutions, if any. 
\end{thm}
\begin{proof}
    By virtue of \cref{thm:Conv-ND},
    $\displaystyle\liminf_{k\to \infty} f(x_{k})= \lim_{k\to \infty} f^{*}_{k} = f^{*},$
    which ensures that
    $\displaystyle\lim_{j\to \infty} f(x_{k_{j}}) = f^{*},$
    for some subsequence $\seqj{x_{k_j}}$ of $\seq{x_{k}}$.
    Furthermore, it follows from sharpness error bound inequality that $\dist(x_{k_j};\X^*)\to 0$ as $j\to\infty$.
    Let us fix $j\in \N_0$. By \cref{lem:Basic1}, we obtain $\seq{x_k}\subseteq \T_\beta$ and
    \begin{align*}
        \dist^2(x_{k};\mathcal{X}^*) &\leq \dist^2(x_{k-1};\mathcal{X}^*) -\tfrac{\alpha_{k-1}}{L}\left(f(x_{k-1})-f^{*}\right)+\alpha_{k-1}^{2}
        \leq \dist^2(x_{k-1};\mathcal{X}^*) +\alpha_{k-1}^{2}\\
        &\le \dist^2(x_{k-2};\mathcal{X}^*) +\alpha_{k-2}^{2}+\alpha_{k-1}^{2}
        \leq \ldots \leq \dist^2(x_{k_j};\mathcal{X}^*) +\sum_{i=k_j}^{k-1} \alpha_{i}^{2}\\[-3mm]
        &\le \dist^2(x_{k_j};\mathcal{X}^*) +\sum_{i=k_j}^{\infty} \alpha_{i}^{2},\quad\quad \forall k> k_{j},
    \end{align*}
    ensuring $\dist(x_{k};\X^*)\to 0$ as $k\to\infty$ inasmuch as $\sum_{i=0}^{\infty} \alpha_{i}^{2}<\infty$. The desired convergence of the objective values and the characterization of the cluster points are immediate consequences of \Cref{lem:ConvAna}~\ref{lem:ConvAna:1} and \ref{lem:ConvAna:2}, completing the proof.
\end{proof}

In what follows, we investigate the convergence properties of the PSGA under two structured step-size policies: diminishing step-size and geometrically decaying (GD) schemes.
We first consider a diminishing step-size given by $\alpha_{k}=\lambda (k+k_0)^{-r}$ with constant parameters $\lambda,k_0>0$ and $0<r<1$.
Such step-size schedules appear frequently in both deterministic frameworks and stochastic optimization algorithms.
The next theorem establishes a sublinear convergence rate for the PSGA under this class of step-sizes.

\begin{thm}[{\bf Convergence rate of diminishing PSGA}]\label{thm:DS}
    Let $r\in (0,1)$ and $\lambda >0$.
    Set
    $$
    A:=2^{r}\lambda~\sqrt{\tfrac{1}{(1-\beta)^2\tau^2}+1},\quad \text{and}\quad k_0:=\max\left\{ \tfrac{2r A}{(1-\beta)\lambda\tau}, A^{\tfrac{1}{r}}\Big(\tfrac{\rho G_{h}}{\beta\mu}\Big)^{\tfrac{1}{r\nu}}, 1\right\}.
    $$
    Let the sequence $\seq{x_{k}}$ be generated by PSGA with SSN step-size $\alpha_{k}=\lambda (k+k_0)^{-r}$ and the initial point $x_0\in \T_\beta$ satisfying
    \(\dist(x_0;\X^*)\le Ak_{0}^{-r}\). Then, the following inequalities hold:
    \begin{equation}\label{eq:DS-i}
        \dist(x_{k} ; \mathcal{X}^*) \leq \tfrac{A}{(k+k_0)^{r}},
    \quad\quad\text{and}\quad\quad
        f(x_{k})-f^* \leq \tfrac{LA}{(k+k_0)^{r}}.
    \end{equation}
\end{thm}

\begin{proof}
    Let us proceed first inequality in \cref{eq:DS-i} by induction. 
    The base case \(k=0\) holds by the initialization hypothesis.  Now suppose that for some \(k\ge0\) the inductive hypothesis
    \(
    \dist(x_k;\mathcal{X}^*) \le A(k+k_0)^{-r}
    \)
    is true.
    Then,
    \[
    \dist(x_{k} ; \mathcal{X}^*) \leq \tfrac{A}{(k+k_0)^{r}} \le \tfrac{A}{k_0^{r}} \le \Big(\tfrac{\beta\mu}{\rho G_{h}}\Big)^{\nicefrac{1}{\nu}},
    \]
    i.e., $x_k\in \T_\beta$.
    It follows from \cref{lem:Basic1} and $\|\zeta_k\|\le L$ that
    \begin{equation}\label{eq:D-1}
        \dist^2(x_{k+1};\mathcal{X}^*) \leq \dist^2(x_{k};\mathcal{X}^*) -\tfrac{2(1-\beta)\tau\lambda}{(k+k_0)^{r}} \dist(x_{k};\mathcal{X}^*) + \tfrac{\lambda^2}{(k+k_0)^{2r}}.
    \end{equation}
    Let us consider the index set
    $$
    I:=\left\{i\in \mathbb{N} :~ \dist(x_{i};\mathcal{X}^*) \leq \tfrac{\lambda}{(1-\beta)\tau(i+k_0)^{r}}\right\}.
    $$
    There are three possible cases: $(i)$ $k+1\in I$; $(ii)$ $k+1\not\in I$ and $k\in I$; $(iii)$ $k, k+1\not \in I$.
    
    In Case~$(i)$, since $\tfrac{\lambda}{(1-\beta)\tau} < A$, clearly \(
    \dist(x_{k+1};\mathcal{X}^*) \le A(k+1+k_0)^{-r}
    \).
    Alternatively, in Case~$(ii)$, $k+1\not\in I$ and $k\in I$, from \cref{eq:D-1} we deduce that
    $$
        \dist^2(x_{k+1};\mathcal{X}^*) \leq \dist^2(x_{k};\mathcal{X}^*)  + \tfrac{\lambda^{2}}{(k+k_0)^{2r}}\leq \tfrac{\lambda^2}{(1-\beta)^2\tau^2(k+k_0)^{2r}} + \tfrac{\lambda^{2}}{(k+k_0)^{2r}} =\tfrac{A^2}{2^{2r}(k+k_0)^{2r}}
        \leq \tfrac{A^{2}}{(k+1+k_0)^{2r}}.
    $$
    Now turning to Case~$(iii)$ where $k+1, k\not\in I$, i.e., \( \tfrac{\lambda}{(k+k_0)^{r}}<(1-\beta)\tau\dist(x_{k};\mathcal{X}^*)\). Substituting this inequality into \cref{eq:D-1} results in
    \begin{align*}
        \dist^2(x_{k+1};\mathcal{X}^*) &< \dist^2(x_{k};\mathcal{X}^*) -\tfrac{2(1-\beta)\tau\lambda}{(k+k_0)^{r}} \dist(x_{k};\mathcal{X}^*) + \tfrac{(1-\beta)\tau\lambda}{(k+k_0)^{r}}\dist(x_{k};\mathcal{X}^*)\\
        &= \dist^2(x_{k};\mathcal{X}^*) -\tfrac{(1-\beta)\tau\lambda}{(k+k_0)^{r}} \dist(x_{k};\mathcal{X}^*).
    \end{align*}
    The function
    $\varphi(t) := t^2 - {(1-\beta)}\tau\lambda(k+k_0)^{-r}  t$ is a convex function
    over $[0,A(k+k_0)^{-r}]$ attaining its maximum at $A(k+k_0)^{-r}$.
    From the former inequality together with the inductive assumption $\dist(x_{k};\mathcal{X}^*)\le A(k+k_0)^{-r}$, we obtain that
    \begin{equation*}
        \dist^2(x_{k+1};\mathcal{X}^*) 
        < \varphi\big(\dist(x_{k};\mathcal{X}^*)\big)
        \leq \varphi\big(A(k+k_0)^{-r}\big)
        =
        \tfrac{A^{2}}{(k+k_0)^{2r}}- \tfrac{(1-\beta) \tau\lambda A} {(k+k_0)^{2r}} \le 
        \tfrac{A^{2}}{(k+k_0)^{2r}} - \tfrac{2r A^2} {(k+k_0)^{2r+1}} \le \tfrac{A^{2}}{(k+1+k_0)^{2r}},
    \end{equation*}
    where the third inequality comes from the lower bound of $k_0$ and for the last inequality we use the convexity of the function $t\mapsto \tfrac{A^{2}}{t^{2 r}}$ on positive real number set.\\
    The second inequality in \cref{eq:DS-i} follows from the first one and \Cref{lem:ConvAna}~\ref{lem:ConvAna:1},
    completing the proof.
\end{proof}

We conclude the convergence analysis by considering a geometrically decaying (GD) step-size,
\[
\alpha_k=\lambda q^k,\qquad \lambda>0,~~0<q<1.
\]
Unlike the diminishing step-size rule considered previously, the GD rule preserves the vanishing property of the step-size while yielding linear convergence. The following theorem establishes linear convergence of the projected subgradient method in terms of the distance to the solution set, the objective value gap, and the iterates.

\begin{thm}[{\bf Convergence rate of GD PSGA}]\label{thm:GDS}
    Let $\beta\in [1-\nicefrac{1}{\sqrt{2}},1)$, $\gamma\in (0,1)$, and $0<\lambda\le (1-\beta)\tau \Big(\tfrac{\beta\mu}{\rho G_{h}}\Big)^{\nicefrac{1}{\nu}}$. Set
    $$
    q:=\sqrt{1-\gamma(1-\beta)^2 \tau^{2}},\quad\quad
    \text{and} \quad\quad
    A:=\max\left\{\tfrac{\lambda}{(1-\beta)\tau},\dist(x_0;\mathcal{X}^*)\right\}.
    $$
    Let the sequence $\seq{x_{k}}$ be generated by PSGA with GD step-size $\alpha_{k}=\lambda q^{k}$ and the initial point $x_{0} \in \mathcal{T}_\beta$ satisfying
    \(\dist(x_0;\mathcal{X}^*)\leq\tfrac{\lambda}{\tau(1-\beta) - \sqrt{\tau^{2}(1-\beta)^2-(1-q^{2})}}\). Then, the following inequalities hold:
    \begin{equation}\label{eq:GDS-i}
        \dist(x_{k} ; \mathcal{X}^*) \leq A q^{k},
        \quad\quad f(x_{k})-f^* \leq LA q^{k},\quad\quad\|x_k-x^*\|\le \tfrac{\lambda q^k}{1-q},
    \end{equation}
    where $x^*$ denotes the limiting point of $\seq{x_k}$.
\end{thm}

\begin{proof}
    By \cref{lem:tau}, we have $\tau\in (0,1]$.
    Moreover, $0<q<1$ and \((1-\beta)^2\tau^{2}-(1-q^{2})>0\), i.e., the upper bound of \(\dist(x_0;\mathcal{X}^*)\) is well-defined.
    Let us verify the first inequality in \cref{eq:GDS-i} by induction.
    Clearly, $\dist(x_{0}; \mathcal{X}^*) \leq A=Aq^0$.
    Assuming that this inequality holds for $k$, we show it for $k+1$.
    Using the upper bound of $\lambda$ together with $x_0\in \mathcal{T}_\beta$, it holds that
    $$
    \dist(x_{k} ; \mathcal{X}^*) \leq A q^{k}\leq A =\max\left\{\tfrac{\lambda}{(1-\beta)\tau},\dist(x_0;\mathcal{X}^*)\right\}\leq \Big(\tfrac{\beta\mu}{\rho G_{h}}\Big)^{\nicefrac{1}{\nu}},
    $$
    i.e., $x_{k}\in \mathcal{T}_\beta$, and consequently, $\|\zeta_k\|\le L$. Hence,
    $$
    \dist^2(x_{k+1};\mathcal{X}^*) \leq \dist^2(x_{k};\mathcal{X}^*)- 2(1-\beta)\tau\lambda q^{k}  \dist(x_{k};\mathcal{X}^*)  + \lambda^{2} q^{2k},
    $$
    due to \cref{lem:Basic1}~\ref{lem:Basic1:1}.
    Noting that $\dist(x_{k};\mathcal{X}^*)\le A q^{k}$ (inductive assumption) and that the quadratic function
    \(\varphi(t) := t^{2} - 2(1-\beta)\tau\lambda q^{k}t + \lambda^{2} q^{2k}\) is convex on $[0,Aq^{k}]$, attaining its maximum at $0$ or $Aq^{k}$, we obtain
    $$
        \dist^2(x_{k+1};\mathcal{X}^*) \leq \varphi \big(\dist(x_{k};\mathcal{X}^*)\big) \leq \max \left\{\varphi(0), \varphi(A q^{k})\right\}
         = q^{2k}\Big(
        A^{2} - 2(1-\beta)\tau\lambda A   + \lambda^{2}\Big),
    $$
    where the equality comes from \(A \ge \tfrac{\lambda}{(1-\beta)\tau}\ge 2(1-\beta)\tau\lambda\) using the lower bound of $\beta$.
    Thus, to ensure $\dist^2(x_{k+1};\mathcal{X}^*)\leq A^{2} q^{2(k+1)}$, it suffices to check
    \begin{equation}\label{eq:GDS-1}
        A^{2}-  \lambda \tau A   + \lambda^{2} \leq A^{2}q^{2}.
    \end{equation}
    Let us consider the quadratic equation \(\left(1  -q^{2}\right)Z^{2} - 2(1-\beta)\tau\lambda Z  + \lambda^{2} =0,\)
    whose positive roots are
    $$\tfrac{\lambda\tau(1-\beta) \pm \lambda\sqrt{ (1-\beta)^2\tau^{2} -(1-q^{2})}}{1  -q^{2}}=
    \tfrac{\lambda}{\tau(1-\beta) \mp \sqrt{\tau^{2}(1-\beta)^2-(1-q^{2})}}.$$
    Using the upper bound of $\dist(x_0;\mathcal{X}^*)$, we come to
    $$\tfrac{\lambda}{\tau(1-\beta) + \sqrt{\tau^{2}(1-\beta)^2-(1-q^{2})}} \leq \tfrac{\lambda}{\tau(1-\beta)}\leq A=\max\left\{\tfrac{\lambda}{(1-\beta)\tau},\dist(x_0;\mathcal{X}^*)\right\}\le \tfrac{\lambda}{\tau(1-\beta) - \sqrt{\tau^{2}(1-\beta)^2-(1-q^{2})}},$$
    which guarantees \cref{eq:GDS-1}.
    Therefore, the inductive step is complete, adjusting the first inequality in~\cref{eq:GDS-i}.\\
    The second inequality in \cref{eq:GDS-i} follows directly from the first one combined with \Cref{lem:ConvAna}~\ref{lem:ConvAna:1}.\\
    Finally, regarding the third inequality in~\cref{eq:GDS-i}, by \Cref{lem:ConvAna}~\ref{lem:ConvAna:3}, the sequence $\seq{x_k}$ converges to a global optimal point $x^*\in \X^*$ and
    $$\|x_k-x^*\| \le \sum_{i\ge k} \|x_{i}-x_{i+1}\| \le \sum_{i\ge k}\alpha_i = \sum_{i\ge k} \lambda q^i = \tfrac{\lambda q^k}{1-q},$$
    validating the desired bound.
\end{proof}

\section{Concluding remarks}\label{sec:conclusion}
In this paper, we introduced the class of relatively weakly convex functions, extending the classical notion of weak convexity through a distance-generating function. We established several fundamental properties of this class, including characterization results, calculus rules, and illustrative examples. We also analyzed its optimization landscape and identified a neighborhood of the global minimizers that was free of saddle points. Motivated by this geometric characterization, we proposed the Projected SubGradient Algorithm (PSGA) together with several step-size strategies. Under a sharpness error bound, we showed that, when initialized within the identified neighborhood, the iterates generated by PSGA converged linearly to a global minimizer. These results broadened the theoretical framework for weakly convex optimization and provided convergence guarantees for first-order methods in the relative setting.


	

	\ifarxiv
		\bibliographystyle{plain}
	\else
		\phantomsection
		\addcontentsline{toc}{section}{References}
		\bibliographystyle{spmpsci}
	\fi
	\bibliography{Bibliography}

\end{document}